\documentclass{amsart}
\usepackage{amssymb,amsthm, amsmath, amsfonts}
\usepackage{graphics}
\usepackage{hyperref}
\usepackage[all]{xy}
\usepackage{enumerate}
\usepackage[mathscr]{eucal}
\usepackage{bbm}

\theoremstyle{plain}
\newtheorem{theorem}{Theorem}[section]
\newtheorem{lemma}[theorem]{Lemma}
\newtheorem{proposition}[theorem]{Proposition}

\newtheorem{corollary}[theorem]{Corollary}
\numberwithin{equation}{section}

\theoremstyle{definition}

\newtheorem{definition}[theorem]{Definition}
\newtheorem{example}[theorem]{Example}
\newtheorem{remark}[theorem]{Remark}

\DeclareMathOperator{\Ob}{Ob}
\DeclareMathOperator{\Module}{-Mod}
\DeclareMathOperator{\lfmod}{-lfmod}
\DeclareMathOperator{\fgmod}{-fgmod}
\DeclareMathOperator{\sfpmod}{-sfpmod}
\DeclareMathOperator{\Hom}{Hom}
\DeclareMathOperator{\Tor}{Tor}
\DeclareMathOperator{\hd}{hd}
\DeclareMathOperator{\gd}{gd}
\DeclareMathOperator{\td}{td}

\DeclareMathOperator{\supp}{supp}
\DeclareMathOperator{\ini}{ini}

\newcommand{\mk}{{\mathbbm{k}}}
\newcommand{\C}{{\mathscr{C}}}
\newcommand{\tC}{{\underline{\mathscr{C}}}}
\newcommand{\Ai}{{\mathsf{A}_{\infty}}}
\newcommand{\Z}{{\mathbb{Z}_+}}
\newcommand{\So}{{S_1}}
\newcommand{\Sa}{{S_a}}
\newcommand{\FI}{{\mathscr{FI}}}
\newcommand{\VI}{{\mathscr{VI_q}}}
\newcommand{\FS}{{\mathscr{FS}}}
\newcommand{\OS}{{\mathscr{OS}}}
\newcommand{\OI}{{\mathscr{OI}}}
\newcommand{\FPi}{{\mathrm{FP}_{\infty}}}
\DeclareMathOperator{\Fq}{\mathbb{F}}

\title[Homological degrees of representations of categories]{Homological degrees of representations of categories with shift functors}
\author{Liping Li}
\address{College of Mathematics and Computer Science, Performance Computing and Stochastic Information Processing (Ministry of Education of China), Hunan Normal University, Changsha, Hunan 410081, P. R. China}
\email{lipingli@hunnu.edu.cn; lixxx480@umn.edu.}
\thanks{This project was supported by the Construct Program of the Key Discipline in Hunan Province.}
\thanks{The motivation of this project originated from numerous discussions with Wee Liang Gan at University of California, Riverside, when the author was a visiting assistant professor there. These discussions motivated the author to study homologies of combinatorial categories in representations stability using the ideas and techniques described in \cite{GL2}. The author would like to thank him for the extremely inspiring communication.}

\begin{document}

\begin{abstract}
Let $\mk$ be a commutative Noetherian ring and $\tC$ be a locally finite $\mk$-linear category equipped with a self-embedding functor of degree 1. We show under a moderate condition that finitely generated torsion representations of $\tC$ are super finitely presented (that is, they have projective resolutions each term of which is finitely generated). In the situation that these self-embedding functors are genetic functors, we give upper bounds for homological degrees of finitely generated torsion modules. These results apply to quite a few categories recently appearing in representation stability theory. In particular, when $\mk$ is a field of characteristic 0, using the result of \cite{CE}, we obtain another upper bound for homological degrees of finitely generated $\FI$-modules.
\end{abstract}

\maketitle

\tableofcontents

\section{Introduction}

\subsection{Background}
Recently, a few combinatorial categories appeared in representation stability theory, a new exciting research area involving many mathematical branches such as representation theory, group cohomology, algebraic topolgy, algebraic geometry, algebraic number theory, commutative algebra, combinatorics, etc. Examples include $\widetilde{\FI}$ \footnote{By $\widetilde{\FI}$ we denote the category of all finite sets and injections, and by $\FI$ we denote a skeletal category of $\widetilde{\FI}$, which only contains objects $[i] = \{1, 2, \ldots, i \}$. By convention, $[0] = \emptyset$.} , the category of finite sets and injections investigated by Church, Ellenberg, Farb, and Nagpal in a series of papers \cite{CEF, CEFN, CE, Farb, N}, and its many variations introduced by Putman, Sam, Snowden, and Wilson in \cite{PS, SS1, SS2, Wilson}. Representation theory of these categories, on one hand was used to prove different stability phenomena as shown in \cite{CF, CEF, CEFN, P, PS}; on the other hand was studied in its own right since these categories have very interesting combinatorial structures, which often induce surprising representational and homological properties. For example, when $\mk$ is a commutative Noetherian ring, the $\mk$-linearizations of many categories are locally Noetherian (\cite{CEF, CEFN, GL1, SS2}); that is, submodules of finitely generated modules are still finitely generated.

A central property shared by many of these combinatorial categories is the existence of a faithful endofunctor called \emph{self-embedding} functor by us; see Definition \ref{self-embedding functors}. It induces a \emph{shift functor} $\Sa$ of degree $a$ in the module categories for every $a \in \Z$, the set of nonnegative integers; see Definition \ref{shift functor}. These self-embedding functors and their induced shift functors have nice properties such as preserving finitely generated projective representations. and hence play an extremely important role in exploring representational and homological properties of these categories. For instances, they were firstly observed in \cite{CEFN} and used to prove the locally Noetherian property of the category $\widetilde{\FI}$ over Noetherian rings $\mk$. When $\mk$ is a field of characteristic 0, Gan and the author used them in \cite{GL2} to show the Koszulity of $\mk$-linearizations of quite a few combinatorial categories simultaneously. In \cite{GL3} we introduced \emph{coinduction functors}, which are right adjoints of shift functors, and gave new and simpler proofs for many results of $\FI$ established in \cite{CEF} and \cite{SS1}. A few months ago, Church and Ellenberg used shift functors to study homologies of $\FI$-modules and proved a surprising upper bound for homological degrees of them (\cite[Theorem A]{CE}).

Motivated by the work of Church and Ellenberg in \cite{CE}, in this paper we focus on homologies of representations of $\mk$-linear categories $\tC$ equipped with self-embedding functors and induced shift functors, where $\mk$ is a commutative ring. Note that $\tC$ in general might not be locally Noetherian. Thus from the homological viewpoint we are more interested in \emph{super finitely presented representations}, or \emph{$\FPi$ representations}, of $\tC$, which by definition have resolutions consisting of finitely generated projective representations. Specifically, we want to know what representations are $\FPi$, explore homologies of these $\FPi$ representations, and estimate upper bounds for their homological degrees.

\subsection{General results}

Let $\mk$ be a commutative Noetherian ring, and let $\tC$ be a \emph{locally finite $\mk$-linear category of type $\Ai$}. That is, objects of $\tC$ are parameterized by nonnegative integers; there is no nonzero morphisms from bigger objects to smaller ones; and $\tC (i, j)$ is a finitely generated $\mk$-module for all $i, j \in \Z$.

By the $\Ai$ structure, $\tC$ has a two-sided ideal $J$ consisting of finite linear combinations of morphisms between distinct objects. Thus we identify $\tC_0$, the set of finite linear combinations of endomorphisms in $\tC$, with $\tC /J$. It has the following decomposition as $\tC$-modules
\begin{equation*}
  \tC_0 = \bigoplus _{i \in \Z} \tC(i,i).
\end{equation*}

Given a $\tC$-module $V$, its $s$-th homology is set to be
\begin{equation*}
H_s(V) = \Tor _s^{\tC} (\tC_0, V), \quad s \in \Z.
\end{equation*}
Since $\tC_0 = \tC / J$ is a $\tC$-bimodule, $H_s(V)$ is a left $\tC$-module as well, and is \textit{torsion} (see Definition \ref{torsion modules}). Moreover, it is \emph{discrete}; that is, the value of $H_s (V)$ on each object is a $\tC$-module concentrated on this object, and $H_s (V)$ is the direct sum of them.

The $s$-th \emph{homological degree} of $V$ is defined to be
\begin{equation*}
\hd_s (V) = \sup \{ i \in \Z \mid \text{the value of $H_s(V)$ on $i$ is nonzero} \}
\end{equation*}
or $-\infty$ if this set is empty. We also define the \emph{torsion degree} of $V$ to be
\begin{equation*}
\td(V) = \sup \{ i \in \Z \mid \Hom_{\tC} (\tC (i, i), V) \neq 0 \}
\end{equation*}
or $-\infty$ if this set is empty. Sometimes we call $\hd_0 (V)$ the \emph{generating degree} and denote it by $\gd(V)$.

\begin{remark} \normalfont
Torsion degrees and homological degrees are closely related. Actually, for $s \geqslant 0$, one can see $\hd_s(V) = \td (H_s (V))$. It is also clear that if $V$ is generated by $\bigoplus _{i=0}^n V_i$, then $n \geqslant \gd(V)$.
\end{remark}

To avoid the situation that certain homological degrees of a $\tC$-module $V$ are infinity, in this paper we mainly consider $\FPi$ modules $V$. It turns out many interesting $\tC$-modules fall into this class when $\tC$ has a self-embedding functor satisfying some assumption.

\begin{theorem} \label{first main result}
Let $\mk$ be a commutative Noetherian ring and $\tC$ be a locally finite $\mk$-linear category of type $\Ai$ equipped with a self-embedding functor $\iota$ of degree 1. Let $\So$ be the induced shift functor. If $\So$ preserves finitely generated projective $\tC$-modules , then a $\tC$-module $V$ is $\FPi$ if and only if so is $\So V$. In particular, every finitely generated torsion module is $\FPi$.
\end{theorem}

\begin{remark} \normalfont
Of course, if $\tC$ is locally Noetherian, the conclusion of this theorem is implied trivially by the Noetherian property. The usefulness of this theorem is that it does not require $\tC$ to be \emph{locally Noetherian}. Indeed, there are many locally finite $\mk$-linear categories of type $\Ai$ which are not locally Noetherian, and in practice it is difficult to check the locally Noetherian property of $\tC$.
\end{remark}

\begin{remark}
A similar result was pointed out by Franjou, Lannes, and Schwartz earlier for the category of vector spaces over finite fields; see \cite[Proposition 10.1]{FLS}.
\end{remark}

Usually it is hard to estimate homological degrees of finitely generated torsion $\tC$-modules $V$. However, if the self-embedding functor is a \emph{genetic functor} (see Definition \ref{genetic functors}, then one can prove that $V$ has Castelnuovo-Mumford regularity (see \cite{ES} for a definition in commutative algebra) expressed in terms of $\td(V)$ only.

\begin{theorem}[Castlenuovo-Mumford regularity] \label{second main result}
Let $\mk$ and $\tC$ be as in the previous theorem, and let $V$ be a $\FPi$ module of $\tC$. Suppose that $\tC$ is equipped with a genetic functor.
\begin{enumerate}
\item If there exists a certain $a \in \Z$ such that
\begin{equation*}
\hd_s (\Sa V) \leqslant \gd (\Sa V) + s,
\end{equation*}
for $s \geqslant 0$, then
\begin{equation*}
\hd_s (V) \leqslant \gd (V) + a + s.
\end{equation*}

\item
If $V$ is a finitely generated torsion module, then
\begin{equation*}
\hd_s (V) \leqslant \td(V) + s.
\end{equation*}
\end{enumerate}
\end{theorem}

\begin{remark} \normalfont
Although in this paper we restrict ourselves to the setting of $\FPi$ modules, with the same essential idea and suitable modifications, the above regularity result holds for all representations generated in finite degrees, which might not be finitely generated. For example, an infinitely generated representation concentrated in one object is generated in finite degrees.
\end{remark}

\begin{remark} \normalfont
We remind the reader that the above result actually can be used to estimate homological degrees of many interesting modules besides torsion modules. Examples include truncations of projective modules, ``syzygies" of finitely generated torsion modules, etc. Moreover, if a $\tC$-module $V$ is \emph{almost isomorphic to} (see Definition \ref{almost isomorphic}) one of the above modules, then homological degrees of $V$ can be estimated as well by using this result.
\end{remark}

\begin{remark} \normalfont
The essential idea of the first statement of this theorem is that while applying the shift functor enough times to a finitely generated representation $V$, it often happens that the shifted module becomes very simple, and hence its homological degrees can be easily estimated. For instance, in a forthcoming paper it will be proved that if we applying the shift functor enough times to an arbitrary finitely generated representation of $\FI$, then higher homologies of the shift module all vanish. Therefore, the condition in the first statement of this theorem is fulfilled trivially.
\end{remark}

\subsection{Application in representation stability theory}

These results immediately apply to a few combinatorial categories in representation stability theory. Let us briefly recall their definitions. For more details, the reader may refer to \cite{SS2} or \cite{GL2}.

\begin{example}[The category $\FI_G$]
Let $G$ be a finite group. The category $\C = \FI_G$ has objects nonnegative integers. For $i, j \in \Z$, $\C(i, j)$ is the set of pairs $(f, g)$ where $f: \{1, \ldots, i \} \to \{1, \ldots, j \}$ is an injection and $g: \{ 1, \ldots, i \} \to G$ is an arbitrary map. For $(f_1, g_1) \in \C(i, j)$ and $(f_2, g_2)\in \C (j, k)$, their composition is $(f_3, g_3)$ where
\begin{equation*}
f_3 = f_2 \circ f_1, \quad \text{and } g_3 (r)= g_2 (f_1(r)) \cdot g_1(r)
\end{equation*}
for $1 \leqslant r \leqslant i$.
\end{example}

\begin{example}[The category $\VI$]
Let $\Fq$ be a finite field. The category $\VI$ has objects nonnegative integers. Morphisms from $i$ to $j$ are linear injections from $\Fq^{\oplus i}$ to $\Fq^{\oplus j}$.
\end{example}

\begin{example}[The category $\OI_G$]
As a subcategory of $\FI_G$, $\OI_G$ has the same objects as $\FI_G$. For $x, y\in \Z$, a morphism $(f,c) \in \FI_G (x, y)$ is contained in $\OI_G(x, y)$ if and only if $f$ is increasing.
\end{example}

\begin{example}[The category $\FI_{d}$]
Let $d$ be a positive integer. The category $\FI_{d}$ has objects nonnegative integers. For $x, y \in \Z$, $\FI_d(x,y)$ is the set of all pairs $(f,\delta)$ where $f: [x]\to [y]$ is injective, and $\delta: [y] \setminus \mathrm{Im}(f) \to [d]$ is an arbitrary map. For $(f_1,\delta_1) \in \FI_d (x,y)$ and $(f_2,\delta_2) \in \FI_d(y,z)$, their composition is $(f_3, \delta_3)$ where $f_3 = f_2 \circ f_1$ and
\begin{equation*}
c_3 (m) = \left\{ \begin{array}{ll}
c_1(r) & \mbox{ if } m=f_2(r) \mbox{ for some } r, \\
c_2(m) & \mbox{ else. } \end{array} \right.
\end{equation*}
\end{example}

\begin{example}[The category $\OI_d$]
As a subcategory of $\FI_d$, $\OI_d$ has the same objects. For $x, y \in \Z$, $\OI_d (x,y)$ consists of pairs $(f, \delta)$ such that $f$ is increasing.
\end{example}

\begin{example}[The opposite category $\FS_G^{\mathrm{op}}$ of $\FS_G$]
Let $G$ be a finite group. The category $\FS_G$ has objects all positive integers. For two objects $x$ and $y$, $\FS_G (y,x)$ consists of pairs $(f, c)$ where $f: [y] \to [x]$ is a surjection, and $c: [y] \to G$ is an arbitrary map. For $(f_1, c_1) \in \FS_G (y, x)$ and $(f_2, c_2)\in \FS_G(z,y)$, their composition is $(f_3, c_3) $ where
\begin{equation*}
f_3 = f_1 \circ f_2, \quad \text{and } c_3 (r) = c_1 (f_2(r)) \cdot c_2(r)
\end{equation*}
for $r \in [z]$.
\end{example}

\begin{example}[The opposite category $\OS_G^{\mathrm{op}}$ of $\OS_G$]
The subcategory $\OS_G$ has the same objects as $\FS_G$. For two objects $x$ and $y$, $\OS_G (y,x)$ consists of pairs $(f,c) \in \FS_G (y, x)$ where $f$ is an ordered surjection.
\end{example}

It has been shown in \cite{GL2} that $\mk$-linearizations of the above categories all have genetic functors. Thus we have:

\begin{corollary} \label{main corollary}
Let $\tC$ be the $\mk$-linearization of one of the following categories:
\begin{equation*}
\FI_G, \quad \OI_G, \quad \VI, \quad \FI_d, \quad \OI_d, \quad \FS_G^{\mathrm{op}}, \quad \OS_G^{\mathrm{op}},
\end{equation*}
and let $V$ be a finitely generated torsion $\tC$-module. Then for $s \in \Z$,
\begin{equation*}
\hd_s (V) \leqslant \td(V) + s.
\end{equation*}
\end{corollary}

The category $\FI$ has many interesting and surprising representational and homological properties. In particular, when $\mk$ is a field of characteristic 0, finitely generated projective $\FI$-modules are injective as well, and every finitely generated $\FI$-module $V$ has a finite injective resolution; see \cite{SS1, GL3}. Using this result, as well as the upper bound given in \cite[Theorem A]{CE}, we get another upper bound for homological degrees of finitely generated $\FI$-modules $V$. That is:

\begin{theorem} \label{third main result}
Let $\mk$ be a field of characteristic 0, and let $V$ be a finitely generated $\FI$-module. Then for $s \geqslant 1$,
\begin{equation*}
\hd_s (V) \leqslant \max \{2\gd (V) - 1, \, \td(V) \} + s.
\end{equation*}
\end{theorem}

\begin{remark} \normalfont
In \cite{CE} Church and Ellenberg  gave the following upper bounds for homological degrees of $\FI$-modules for an arbitrary ring:
\begin{equation*}
\hd_s (V) \leqslant \gd (V) + \hd_1 (V) + s - 1.
\end{equation*}
Compared to it, the conclusion of Theorem \ref{third main result} has a big shortcoming. That is, it depends on the existence of a finite injective resolution for every finitely generated $\FI$-module. When $\mk$ is an arbitrary commutative Noetherian ring, this fact may no longer be true, even for fields with a positive characteristic. \footnote{Using a crucial technique developed in \cite{N, CE}, in a forthcoming paper we will remove from the above theorem the unnecessary assumption that $\mk$ is a field of characteristic 0.}

But our result does have some advantages. Firstly, in practice it is usually easier to obtain $\td (V)$ compared to $\hd_1 (V)$. Moreover, if $V$ is torsionless, then $\td(V) = 0$, and one deduces
\begin{equation*}
\hd_s (V) \leqslant 2\gd (V) + s - 1
\end{equation*}
for $s \geqslant 1$. Since by Lemma \ref{compare gd to hd} one always has $\hd_1(V) > \gd(V)$ when $k$ is a field of characteristic 0, our bounds are a little more optimal for these modules; see Example \ref{example}.

In \cite{SS1} Sam and Snowden have shown that when $\mk$ is a field of characteristic 0, every finitely generated $\FI$-module has finite Castlenuovo-Mumford regularity; see \cite[Corollary 6.3.5]{SS2}. But an explicit upper bound of this regularity was not given.
\end{remark}

\subsection{Organization}

This paper is organized as follows. In Section 2 we give basic definitions and elementary results used throughout this paper. In particular, \emph{homological degrees}, \emph{torsion degrees}, and \emph{generating degrees} of modules are defined and their relationships are clarified.

General results are described and proved in Sections 3 and 4. In Section 3 we consider self-embedding functors and their induced shift functors. Under a moderate assumption, we show that a representation is finitely generated (resp., finitely presented; $\FPi$) if and only if so is the shifted one. Using this, one can easily deduce Theorem \ref{first main result}. Genetic functors and their induced shift functors are studied in Section 4. We describe a crucial recursive procedure to compare homological degrees of a module to those of the shifted module, and prove Theorem \ref{second main result}.

Applications of general results in representation stability theory are collected in Section 5. Corollary \ref{main corollary} is an immediate result of Theorem \ref{second main result} since the existence of genetic functors for these categories was already proved in \cite{GL2}. Moreover, when $\mk$ is a field of characteristic 0, we give another proof of the Koszulity of these combinatorial categories, and show that the category of Koszul modules is closed under truncation functors (Proposition \ref{Koszul modules}). Finally, using the method described in \cite{GL3}, we explicitly construct an injective resolution for every finitely generated $\FI$-module, and use this resolution, as well as the general results, to establish Theorem \ref{third main result}.

\section{Preliminaries}

Throughout this paper let $\mk$ be a commutative Noetherian ring with identity, and let $\tC$ be a (small) $\mk$-linear category. That is, for $x, y \in \Ob \tC$, the morphism set $\tC (x, y)$ is a $\mk$-module; furthermore, composition of morphisms is $\mk$-linear. Note that for every object $x \in \Ob \tC$, $\tC (x, x)$ is a $\mk$-algebra with identity $1_x$.

\subsection{Type $\Ai$ categories and their representations}

Recall that the $\mk$-linear category $\tC$ is of \emph{type $\Ai$} if $\Ob \tC = \Z$, and $\C (i, j) = 0$ whenever $i > j$. For technical purpose, we suppose that $\tC$ satisfies the following \emph{locally finite} condition: $\C (i, j)$ is a finitely generated $\mk$-module for all $i, j \in \Z$.

A \textit{representation} $V$ of $\tC$ (or a $\tC$-\emph{module}) by definition is a covariant $\mk$-linear functor from $\tC$ to $\mk \Module$, the category of left $\mk$-modules. For each object $i \in \Z$, we let $V_i = V(i)$ be the image of $i $ under $V$, which is a $\C(i, i)$-module, called the \emph{value} of $V$ on $i$.

\begin{remark} \normalfont
Clearly $\tC$ can be viewed as a (non-unital) algebra $A_{\tC}$ in a natural way. Therefore, given a representation $V$ of $\tC$, the $\mk$-module $\bigoplus _{i \in \Z} V_i$ is an $A_{\tC}$-module, denoted by $V$ again by abuse of notation. The category of representations of $\tC$ can be identified with a full subcategory of $A_{\tC}$-modules. That is, an $A_{\tC}$-module $V$ is a representation of $\tC$ if and only if $V \cong \bigoplus _{i \in \Z} 1_i V$ as $\mk$-modules. Sometimes we regard $\tC$ as an algebra via identifying it with $A_{\tC}$ and hopefully this will not cause confusion to the reader.
\end{remark}

In this paper we only consider representations of $\tC$, or $\tC$-modules, rather than all $A_{\tC}$-modules. Denote by $\tC \Module$ the category all $\tC$-modules, and by $\tC \lfmod$ the category of \emph{locally finite} $\tC$-modules. We remind the reader that a $\tC$-module $V$ is locally finite if for each $i \in \Z$, the $\tC(i,i)$-module $V_i$ restricted as a $\mk$-module, is finitely generated. Since $\mk$ is Noetherian, and kernels and cokernels of $\tC$-modules are defined via values on objects, the category $\tC \lfmod$ is abelian.

\subsection{Finitely generated modules}

For $i \in \Z$, the representable functor $\tC (i, -)$ is a projective object in $\tC \Module$. We identify it with the $\tC$-module $\tC 1_i$ consisting of finite linear combinations of morphisms starting from the object $i$.

\begin{lemma} \label{enough projective}
The category $\tC \lfmod$ has enough projectives.
\end{lemma}

\begin{proof}
Indeed, for $V \in \tC \lfmod$, one has $V = \bigoplus _{i \in \Z} V_i$ such that the $\tC(i,i)$-module $V_i$ restricted as a $\mk$-module is finitely generated. Therefore, as a $\tC (i, i)$-module it is finitely generated as well. Now choose a surjection $\tC (i, i) ^{\oplus a_i} \to V_i$ of $\tC(i, i)$-modules for each $i \in \Z$. We get a surjection
\begin{equation*}
\bigoplus _{i \in \Z} \tC (i, i) ^{\oplus a_i} \to \bigoplus _{i \in \Z} V_i,
\end{equation*}
which induces a surjection of $\tC$-modules as follows:
\begin{equation*}
P = \bigoplus _{i \in \Z} \tC \otimes _{\tC (i, i)} \tC (i, i) ^{\oplus a_i} \cong \bigoplus _{i \in \Z} (\tC 1_i) ^{\oplus a_i} \to V.
\end{equation*}

One has to show that $P$ is locally finite. But this is clear since for every $j$,
\begin{equation*}
P_j = \bigoplus _{0 \leqslant i \leqslant j} \tC (i, j) ^{\oplus a_i}
\end{equation*}
which is a finitely generated $\mk$-module by the locally finite condition of $\tC$.
\end{proof}

Based on this lemma, one can define finitely generated $\tC$-modules.

\begin{definition}
A $\tC$-module $V$ is finitely generated if there exists a surjective $\tC$-module homomorphism
\begin{equation*}
\pi: \bigoplus _{i \in \Z} (\tC 1_i) ^{\oplus a_i} \to V
\end{equation*}
such that $\sum _{i \in \Z} a_i < \infty$.
\end{definition}

Finitely generated $\tC$-modules are always locally finite. Therefore, the category of finitely generated $\tC$-modules, denoted by $\tC \fgmod$, is a full subcategory of $\tC \lfmod$. However, $\tC \fgmod$ in general is not abelian, and it is abelian if and only if $\tC$ is a \emph{locally Noetherian} category; i.e., submodules of $\tC 1_i$ are finitely generated for $i \in \Z$.

\begin{remark} \normalfont \label{generated in finite degrees}
A locally finite $\tC$-module $V$ is finitely generated if and only if it is generated in finite degrees. That is, $V$ is generated by the subset $\bigoplus _{i = 0}^N V_i$ for a certain $N \in \Z$; or equivalently, any submodule of $V$ containing $\bigoplus _{i = 0}^N V_i$ coincides with $V$.
\end{remark}

\subsection{Truncations}

Given $n \in \Z$, one defines the truncation functor of degree $n$
\begin{equation*}
\tau_n: \tC \Module \to \tC \Module, \quad V \mapsto \tau_n V = \bigoplus _{i \geqslant n} V_i.
\end{equation*}

Note that $\tau_n V$ is viewed as a $\tC$-module via setting its value on each $i$ with $i < n$ to be 0. Moreover, $\bigoplus _{i < n} V_i$ is also a $\tC$-module via identifying it with the quotient module $V / \tau_n V$. Clearly, $\tau_n$ is an exact functor, and one has
\begin{equation*}
\Hom _{\tC} (\tau_n V, W) \cong \Hom_{\tC} (\tau_n V, \tau_n W)
\end{equation*}
for $\tC$-modules $V$ and $W$.

The truncation functor $\tau_n$ preserves locally finite property, and hence induces a functor $\tC \lfmod \to \tC \lfmod$ which is still denoted by $\tau_n$. However, it does not preserve finitely generated property as shown by the following example.

\begin{example} \normalfont
Let $\tC$ be the $\mk$-linearization of the following quiver
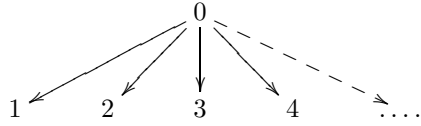
\begin{equation*}
\xymatrix{
 & & 0 \ar[lld] \ar[ld] \ar[d] \ar[rd] \ar@{-->}[rrd] & & \\
1 & 2 & 3 & 4 & \ldots.
}
\end{equation*}
The reader can check that $\tC$ is a locally finite $\mk$-linear category of type $\Ai$. However, $\tau_1 (\tC 1_0)$ is not finitely generated.
\end{example}

\subsection{Torsion degrees and torsion modules}

Let $V$ be a locally finite $\tC$-module. The \emph{torsion degree} of $V$, denoted by $\td(V)$, is defined to be
\begin{equation*}
\td(V) = \sup \{ i \in \Z \mid \Hom _{\tC} (\tC(i,i), V) \neq 0 \},
\end{equation*}
where $\tC(i, i)$ is viewed as a $\tC$-module in a natural way. If the above set is empty, we set $\td(V) = -\infty$, and say that $V$ is \emph{torsionless}.

\begin{lemma} \label{compare td}
Let $0 \to U \to V \to W \to 0$ be a short exact sequence of locally finite $\tC$-modules. Then
\begin{equation*}
\td(U) \leqslant \td(V) \leqslant \max \{ \td(U), \td(W) \}.
\end{equation*}
\end{lemma}

\begin{proof}
It $V$ is torsionless, then $\Hom _{\tC} (\tC(i,i), V) = 0$ for every $i \in \Z$, and hence $\Hom _{\tC} (\tC(i,i), U) = 0$ as well. Therefore, both $\td(U)$ and $\td(V)$ are $-\infty$, and the conclusion holds.

If $\td(V) = \infty$, then for every $N \in \Z$, one can find $i \in \Z$ with $i > N$ such that $\Hom _{\tC} (\tC(i,i), V)$ is nonzero. Applying the functor $\Hom _{\tC} (\tC(i,i), -)$ to the short exact sequence one deduces that either $\Hom _{\tC} (\tC(i,i), U)$ is nonzero or $\Hom _{\tC} (\tC(i,i), W)$ is nonzero. Consequently, either $\td(U)$ or $\td(V)$ is $\infty$. The conclusion still holds.

If $\td(V)$ is a finite number, we can let $s$ be an integer with $i > \td(V)$. By definition, $\Hom _{\tC} (\tC (i,i), V)$ is 0. Applying $\Hom _{\tC} (\tC(i,i), -)$ to the sequence one deduces $\Hom _{\tC} (\tC (i,i), U) = 0$, so $\td(U) \leqslant \td(V)$. To check the second inequality, one only needs to note that if both $\Hom _{\tC} (\tC (i,i), U)$ and $\Hom _{\tC} (\tC (i,i), W)$ are 0, then $\Hom _{\tC} (\tC (i,i), V)$ must be 0 as well.
\end{proof}

\begin{definition} \label{torsion modules}
A finitely generated $\tC$-module $U$ is a torsion module if there exists a certain $N \in \Z$ such that $U_i = 0$ for $i > N$. A locally finite $\tC$-module $V$ is a torsion module if $V$ can be written as a direct sum of finitely generated torsion $\tC$-modules.
\end{definition}

\begin{remark} \normalfont
Note that for an infinitely generated torsion module $V$, one may not be able to find a fixed number $N \in \Z$ such that $V_i = 0$ for $i > N$. However, if $V$ is a (nonzero) finitely generated torsion module, then one can let $N = \td(V)$, which is nothing but the last object on which the value of $V$ is nonzero.
\end{remark}

\subsection{Homologies of representations}

Let
\begin{equation*}
J = \bigoplus _{0 \leqslant i < j < \infty} \tC (i, j),
\end{equation*}
which is a two-sided ideal of $\tC$ (or more precisely, a two-sided ideal of the $\mk$-algebra $A_{\tC}$). Let
\begin{equation*}
\tC_0 = \bigoplus _{i \in \Z} \tC(i, i),
\end{equation*}
which is a (left and right) quotient $\tC$-module via identifying it with $\tC / J$.

For every $V \in \tC \lfmod$, the map
\begin{equation*}
V \mapsto V/JV \cong \tC_0 \otimes _{\tC} V
\end{equation*}
gives rise to a functor from $\tC \lfmod$ to itself. This is a right exact functor, so we define \emph{homologies} of $V$ by setting
\begin{equation*}
H_i (V) = \Tor _i^{\tC} (\tC_0, V),
\end{equation*}
which are $\tC$-modules again. Calculations of homologies of $V$ can be carried out via the usual homological method. That is, take a projective resolution of $V$ and tensor it with $\tC_0 \otimes _{\tC} -$.

\begin{remark} \normalfont
This definition is motivated by the definition of homologies of $\FI$-modules discussed in literatures such as \cite{CE, CEF, CEFN, GL4}. In \cite{CE, GL4}, it was pointed out that homologies of FI-modules can be computed through an explicit complex constructed by using the shift functor.
\end{remark}

\begin{remark} \normalfont \label{discrete structure of homologies}
From the above definition one knows that for each $i \in \Z$, $H_i(V)$ has the following decomposition as $\tC$-modules
\begin{equation*}
H_i(V) \cong \bigoplus _{j \in \Z} H_i(V)_j,
\end{equation*}
where $H_i(V)_j$ is a $\tC$-module concentrated on object $j$. In particular, $H_i(V)$ is a torsion module.
\end{remark}

The \emph{homological degrees} of $V$ are defined via letting
\begin{equation*}
\hd_i (V) = \td (H_i(V)), \, i \in \Z.
\end{equation*}
For $i = 0$, we call $\hd_0 (V)$ the \emph{generating degree} of $V$, denoted by $\gd(V)$. To justify this name, one only needs to keep in mind that $\gd(V)$ has the following interpretation: for a nonzero module $V$,
\begin{align*}
\gd(V) & = \sup \{i \in \Z \mid (V/JV)_i \neq 0 \}\\
& = \min \{ N \in \Z \cup \{\infty \} \mid V \text{ is generated by } \bigoplus _{i \leqslant N} V_i \}.
\end{align*}

The following result is trivial.

\begin{lemma} \label{compare gd}
Let $0 \to U \to V \to W \to 0$ be a short exact sequence of locally finite $\tC$-modules. Then
\begin{equation*}
\gd(W) \leqslant \gd(V) \leqslant \max \{\gd(U), \, \gd(W) \}.
\end{equation*}
\end{lemma}

\section{Super finitely presented property}

In this section we focus on categories equipped with self-embedding functors, and consider $\FPi$ modules of these categories.

\subsection{Definitions} \label{self-embedding functors}

An endofunctor $\iota: \tC \to \tC$ is called a \emph{self-embedding} functor of degree 1 if $\iota$ is faithful and one has $\iota (s) = s+1$ for $s \in \Z$. It induces a pull-back functor $\iota^{\ast}: \tC \Module \to \tC \Module$. Explicitly, if $V: \tC \to \mk \Module$ is a representation of $\tC$, then one defines $\iota ^{\ast} (V) = V \circ \iota$.

\begin{remark} \normalfont
Clearly, $\iota^{\ast}$ restricts to a functor $\tC \lfmod \to \tC \lfmod$. We denote this restricted functor $\iota^{\ast}$ as well. However, it is not clear whether $\iota ^{\ast}$ preserves finitely generated $\tC$-modules.
\end{remark}

\begin{definition} \label{shift functor}
Suppose that $\tC$ has a self-embedding functor $\iota$ of degree 1. The shift functor $\So$ of degree 1 is defined to be $\iota^{\ast} \circ \tau_1$, where $\tau_1$ is the truncation functor of degree 1. For $a \geqslant 1$, one can define $\Sa = \So \circ S_{a-1}$ recursively and call it the \emph{shift functor} of degree $a$.
\end{definition}

\begin{remark} \normalfont
When $\tC$ is the $\mk$-linearization of $\FI$, one readily sees that shift functors defined in our sense are precisely shift functors $\Sa$ introduced in \cite{CEF, CEFN}. Moreover, it is also clear that $\Sa = (\iota^{\ast})^a \circ \tau_a$, and all $\Sa$ are exact functors.
\end{remark}

In the rest of this paper we suppose that $\tC$ is equipped with a fixed self-embedding functor $\iota$ of degree 1, and fix $\So$ to be the corresponding shift functor of degree 1. Moreover, we assume that $\So$ satisfies the following property:
\begin{description}
\item[FGP] For every $s \in \Z$, $\So (\tC 1_s)$ is a finitely generated projective $\tC$-module.
\end{description}
In other words, $\So$ preserves finitely generated projective $\tC$-modules.

\subsection{Finitely generated property}

As the starting point, we show that shift functors restrict to endofunctors in $\tC \fgmod$.

\begin{lemma} \label{shift keeps fg}
A locally finite $\tC$-module $V$ is finitely generated if and only if so is $\Sa V$ for a certain $a \in \Z$.
\end{lemma}

\begin{proof}
The conclusion holds trivially for $a = 0$. We prove the conclusion for $a = 1$, since for an arbitrary $a \geqslant 1$, the conclusion follows from recursion.

If $V$ is finitely generated, by definition, one can find a surjective $\tC$-module homomorphism
\begin{equation*}
P = \bigoplus _{i \in \Z} (\tC 1_i) ^{\oplus a_i} \to V
\end{equation*}
such that $\sum _{i \in \Z} a_i < \infty$. Applying $\So$ to this surjection, one deduces a surjection $\So P \to \So V$. By the FGP condition, $\So P$ is a finitely generated $\tC$-module, so is its quotient $\So V$.

Now suppose that $\So V$ is finitely generated. By Remark \ref{generated in finite degrees}, we note that $\So V$ is generated in degrees $\leqslant N$ for a certain $N \in \Z$. Therefore, for any $i \geqslant s$, it is always true that
\begin{equation*}
\sum _{i \leqslant N} \tC (i, s) \cdot (\So V)_i = (\So V)_s.
\end{equation*}
But $(\So V)_i = V_{i+1}$, and $\iota$ identify $\tC(i,s)$ with the subset $\iota (\tC(i, s)) \subseteq \tC(i+1, s+1)$, this means
\begin{equation*}
\sum _{i \leqslant N} \iota (\tC (i, s)) \cdot V_{i+1} = V_{s+1},
\end{equation*}
and hence
\begin{equation*}
\sum _{i \leqslant N} \tC (i+1, s+1) \cdot V_{i+1} = V_{s+1}.
\end{equation*}
That is, $V$ is generated in degrees $\leqslant N+1$. Consequently, it is finitely generated.
\end{proof}

\begin{remark} \normalfont
From the proof one easily sees that if $\So V$ is finitely generated, then $V$ must be finitely generated as well even if the FGP condition fails. This result actually comes from the existence of a self-embedding functor. Conversely, suppose that $V$ is finitely generated. To show the finite generality of $\So V$, one only needs to assume that every $\So (\tC 1_s)$ is finitely generated, and the projectivity of $\So (\tC 1_s)$ is not required.
\end{remark}

\subsection{Super finitely presented property}

In this subsection we consider super finitely presented modules.

\begin{definition}
A locally finite $\tC$-module $V$ is finitely presented if there is a projective presentation $P^1 \to P^0 \to V \to 0$ such that both $P^1$ and $P^0$ are finitely generated. We say that $V$ is super finitely presente (or $\FPi$) if there is a projective resolution $P^{\bullet} \to V \to 0$ such that every $P^i$ is finitely generated.
\end{definition}

\begin{remark} \label{characterization of fg, fp, and sfp} \textnormal
In the language of homologies, $V$ is finitely generated (resp., finitely presented; $\FPi$) if and only if $H_i (V)$ is finitely generated for $i = 0$ (resp., for $i \leqslant 1$; for $i \in \Z$). Equivalently, $V$ is finitely generated (resp., finitely presented; $\FPi$) if and only if $\hd_i (V) < \infty$ for $i = 0$ (resp, for $i \leqslant 1$; for $i \in \Z$).
\end{remark}

Since $\tC$ might not be locally Noetherian, finitely generated $\tC$-modules in general are not $\FPi$. However, we have the following result.

\begin{proposition} \label{shift keeps sfp}
Let $V$ be a locally finite $\tC$-module. Then $V$ is finitely presented (resp., $\FPi$) if and only if so is $\Sa V$ for a certain $a \in \Z$.
\end{proposition}

\begin{proof}
Again, it is enough to show the conclusion for $a = 1$. We only consider $\FPi$ modules since the same technique applies to finitely presented modules.

If $V$ is $\FPi$, then one can find a projective resolution $P^{\bullet} \to V \to 0$ such that every $P^i$ is a finitely generated $\tC$-module. Since $\So$ is exact and $\tC$ has the FGP property, one gets a projective resolution $\So P^{\bullet} \to \So V \to 0$ such that every $\So P^i$ is still finitely generated. In other words, $\So V$ is $\FPi$.

Now suppose that $\So V$ is $\FPi$. In particular, $\So V$ is finitely generated, so is $V$ by Lemma \ref{shift keeps fg}. Therefore, one gets a short exact sequence
\begin{equation*}
0 \to V^1 \to P^0 \to V \to 0
\end{equation*}
where $P^0$ is a finitely generated projective $\tC$-module. We claim that $V^1$ is finitely generated as well. To see this, it is enough to prove the finitely generality of $\So V^1$. We apply $\So$ to the exact sequence to get
\begin{equation*}
0 \to \So V^1 \to \So P^0 \to \So V \to 0,
\end{equation*}
which gives a long exact sequence
\begin{align*}
\ldots \to H_1 (\So V^1) \to H_1 (\So P^0) = 0 \to H_1 (\So V) \to \\
H_0 (\So V^1) \to H_0 (\So P^0) \to H_0 (\So V) \to 0.
\end{align*}
Since $\So V$ is $\FPi$, by Remark \ref{characterization of fg, fp, and sfp}, both $H_0 (\So V)$ and $H_1 (\So V)$ are finitely generated torsion $\tC$-modules. Clearly, $H_0 (\So P^0)$ is a finitely generated torsion $\tC$-module as well. This forces $H_0 (\So V^1)$ to be finitely generated, so is $\So V^1$. Therefore, as claimed, $V^1$ is finitely generated as well. Moreover, since
\begin{equation*}
H_i (\So V^1) \cong H_{i+1} (\So V)
\end{equation*}
for $i \geqslant 1$, $\So V^1$ is $\FPi$. Replacing $V$ by $V^1$ and using the same argument, one gets a short exact sequence
\begin{equation*}
0 \to V^2 \to P^1 \to V^1 \to 0
\end{equation*}
such that every term is finitely generated. Recursively, we can construct a projective resolution $P^{\bullet}$ for $V$ such that every $P^i$ is finitely generated. That is, $V$ is $\FPi$.
\end{proof}

It immediately implies:

\begin{corollary} \label{torsion modules are sfp}
Every finitely generated torsion $\tC$-module is $\FPi$.
\end{corollary}

\begin{proof}
Let $V$ be a nonzero finitely generated torsion $\tC$-module, and let $a = \td(V) + 1$ which is finite. Then one has $\Sa V = 0$, clearly $\FPi$.
\end{proof}

\subsection{Category of $\FPi$ modules}

Now we consider the category of locally finite $\FPi$ modules, and denote it by $\tC \sfpmod$. Note that $\tC \sfpmod \subseteq \tC \fgmod \subseteq \tC \lfmod$.

The following lemma is well known. For the convenience of the reader, we give a proof using homologies of modules.

\begin{lemma}
Let $0 \to U \to V \to W \to 0$ be a short exact sequence of locally finite $\tC$-modules. If two of them are $\FPi$, so is the third one.
\end{lemma}

\begin{proof}
Applying $\tC_0 \otimes _{\tC} -$ to the short exact sequence one gets the following long exact sequence:
\begin{equation*}
\ldots \to H_2 (W) \to H_1 (U) \to H_1 (V) \to H_1 (W) \to H_0 (U) \to H_0 (V) \to H_0 (W) \to 0.
\end{equation*}
If two of them are $\FPi$, then the homologies of these two modules are only supported on finitely many objects in $\tC$. Therefore, homologies of the third one must be also supported on finitely many objects by the long exact sequence. In other words, every homological degree of this module is finite.
\end{proof}

\begin{definition} \label{almost isomorphic}
Two locally finite $\tC$-modules $U$ and $V$ are almost isomorphic, denoted by $U \sim V$ if there exists a certain $N \in \Z$ such that $\tau_N U \cong \tau_N V$ as $\tC$-modules. This is an equivalence relation.
\end{definition}

Recall that $\tC \sfpmod$ the category of $\FPi$, locally finite $\tC$-modules.

\begin{proposition} \label{equivalent conditions}
The following statements are equivalent:
\begin{enumerate}
\item The category $\tC \sfpmod$ contains all finitely generated torsion modules.
\item The category $\tC \sfpmod$ is closed under the equivalence relation $\sim$. That is, if $U$ and $V$ are locally finite $\tC$-modules such that $U \sim V$, then one is $\FPi$ if and only if so is the other one.
\item A locally finite $\tC$-module $V$ is $\FPi$ if and only if so is a truncation $\tau_i V$ for some $i \in \Z$.
\end{enumerate}
\end{proposition}

\begin{proof}
$(1) \Rightarrow (2)$: Suppose that $U$ is $\FPi$. Since $U \sim V$, there is a certain $N \in \Z$ such that $\tau_N U \cong \tau_N V$. Now consider the exact sequence
\begin{equation*}
0 \to \tau_N U \to U \to \overline{U} \to 0.
\end{equation*}
Clearly, $\overline{U}$ is a finitely generated torsion module, and hence is $\FPi$ by the assumption. By the previous lemma, $\tau_N U$ is $\FPi$ as well.

Now in the exact sequence
\begin{equation*}
0 \to \tau_N V \to V \to \overline{V} \to 0
\end{equation*}
$\tau_N V \cong \tau_N U$, and $\overline{V}$ is a finitely generated torsion module. Since they both are $\FPi$, so is $V$ again by the previous lemma.

$(2) \Rightarrow (3)$: Notes that $V \sim \tau_i V$.

$(3) \Rightarrow (1)$: If $V$ is a finitely generated torsion module, then for a large enough $i$, one has $\tau_i V = 0$, which is clearly contained in $\tC \sfpmod$. Therefore, $V$ is contained in $\tC \sfpmod$ as well.
\end{proof}

\begin{remark} \normalfont
Actually, the previous proposition holds for an arbitrary locally finite $\mk$-linear category $\tC$ even if it does not have a self-embedding functor.
\end{remark}

Now we collect main results in this section in the following theorem.

\begin{theorem}
A locally finite $\tC$-module is $\FPi$ if and only if so is $\So V$. In particular, the following equivalent conditions hold:
\begin{enumerate}
\item The category $\tC \sfpmod$ contains all finitely generated torsion modules.
\item The category $\tC \sfpmod$ is closed under the equivalence relation $\sim$.
\item A locally finite $\tC$-module $V$ is $\FPi$ if and only if so is a truncation $\tau_i V$ for some $i \in \Z$.
\end{enumerate}
\end{theorem}

\begin{proof}
The conclusion follows from Proposition \ref{shift keeps sfp}, Corollary \ref{torsion modules are sfp}, and Proposition \ref{equivalent conditions}.
\end{proof}

\section{Upper bounds of homological degrees}

As before, let $\tC$ be a locally finite $\mk$-linear category of type $\Ai$ equipped with a self-embedding functor $\iota$ of degree 1. Let $\So$ be the shift functor induced by $\iota$. In the previous section we have shown that many interesting $\tC$-modules including finitely generated torsion modules are $\FPi$ provided that $\So$ preserves finitely generated projective $\tC$-modules (the FGP condition). Therefore, we may try to compute their homological degrees. In general it is very difficult to get an explicit answer for this question. However, for many combinatorial categories appearing in representation stability theory, their self-embedding functors and the induced shift functors have extra interesting properties, allowing us to get upper bounds for homological degrees of finitely generated torsion modules.

\subsection{Genetic functors}

\emph{Genetic functors} were firstly introduced and studied in \cite{GL2}, where they were used to show the Koszulity of many categories in representation stability theory.

\begin{definition} \label{genetic functors}
A self-embedding functor $\iota: \tC \to \tC$ of degree 1 is a genetic functor if the corresponding shift functor $\So$ satisfies the FGP condition, and moreover $\gd (\So (\tC 1_s)) \leqslant s$ for $s \in \Z$.
\end{definition}

\begin{remark} \normalfont
The first example of categories equipped with genetic functors is the category $\FI$ observed by Church, Ellenberg, and Farb in \cite{CEF}. They explicitly constructed the shift functor and proved that it has the property specified in the above definition, but did not mention that this shift functor is induced by a genetic functor.
\end{remark}

From the definition, one immediately observes that $\So (\tC 1_s) = P_{s-1} \oplus P_s$, where $P_{s-1}$ and $P_s$ are finitely generated projective $\tC$-modules generated in degree $s-1$ and degree $s$ respectively; here for $s = 0$ we let $P_{-1} = 0$.

\subsection{A recursive technique}

The importance of genetic functors and their induced shift functors, is that they give us a recursive way to compute homological degrees. Firstly we consider the zeroth homological degree, and strengthen the conclusion of Lemma \ref{shift keeps fg} as follows.

\begin{lemma} \label{compare h0}
Let $V$ be a finitely generated $\tC$-module. Then one has:
\begin{equation*}
\gd (\So V) \leqslant \gd (V) \leqslant \gd (\So V) + 1.
\end{equation*}
\end{lemma}

\begin{proof}
The second inequality has already been established in the proof of Lemma \ref{shift keeps fg}. To show the first one, we notice that if $V$ is generated in degrees $\leqslant n$, then there is a surjection $P \to V$ such that $P$ is a finitely generated projective $\tC$-module generated in degrees $\leqslant n$. Applying $\So$ we get a surjection $\So P \to \So V$. The property of genetic functors tells us that $\So P$ is still generated in degrees $\leqslant n$.
\end{proof}

Another technical lemma is:

\begin{lemma}
Let $0 \to W \to P \to V \to 0$ be a short exact sequence of locally finite $\tC$-modules. Suppose that $P$ is projective and $\gd(P) = \gd(V)$. Then
\begin{equation*}
\hd_1(V) \leqslant \gd(W) \leqslant \max \{ \gd(V), \, \hd_1(V) \}.
\end{equation*}
\end{lemma}

\begin{proof}
Applying $\tC_0 \otimes _{\tC} -$ to the exact sequence one obtains
\begin{equation*}
0 \to H_1(V) \to H_0(W) \to H_0(P) \to H_0(V) \to 0.
\end{equation*}
Now the conclusion follows from Lemma \ref{compare td}.
\end{proof}

Now let us compare homological degrees of $\FPi$ modules to those of shifted modules.

\begin{proposition} \label{hd under shift}
Let $V$ be a locally finite $\tC$-module.
\begin{enumerate}
\item If $\So V$ is finitely presented, then one has
\begin{equation*}
\hd_1 (V) \leqslant \max \{ \hd_0 (V) + 1, \, \hd_1 (\So V) + 1 \}.
\end{equation*}

\item If $\So V$ is $\FPi$, then one has
\begin{equation*}
\hd_s (V) \leqslant \max \{\hd_0(V) + 1, \, \ldots, \, \hd_{s-1} (V) + 1, \, \hd_s (\So V) + 1 \}
\end{equation*}
for $s \geqslant 0$.
\end{enumerate}
\end{proposition}

Here we set $\hd _{-1} (V) = 0$.

\begin{proof}
The conclusion holds trivially for $V = 0$, so we suppose that $V \neq 0$. Since $\So V$ is finitely presented, there is a short exact sequence
\begin{equation*}
0 \to V^1 \to P \to V \to 0
\end{equation*}
where $P$ is a finitely generated projective $\tC$-module. Clearly, one can assume that $\hd_0 (P) = \hd_0 (V)$. Applying $\So$ to it one has
\begin{equation*}
0 \to \So V^1 \to \So P \to \So V \to 0.
\end{equation*}
Note that $\So P$ is still projective. Applying $\tC_0 \otimes _{\tC} -$ to these two sequences we get
\begin{equation*}
0 \to H_1(V) \to H_0(V^1) \to H_0(P) \to H_0(V) \to 0
\end{equation*}
and
\begin{equation*}
0 \to H_1 (\So V) \to H_0 (\So V^1) \to H_0 (\So P) \to H_0 (\So V) \to 0.
\end{equation*}
Therefore
\begin{align*}
& \hd_1 (V) = \td (H_1(V)) \leqslant \td (H_0 (V^1)) & \text{ by Lemma \ref{compare td}} \\
& = \gd (V^1) \leqslant \gd (\So V^1) + 1 & \text{ by Lemma \ref{compare h0}}\\
& = \td (H_0(\So V^1)) + 1\\
& \leqslant \max \{ \td (H_1 (\So V)) + 1, \, \td(H_0 (\So P)) + 1 \} & \text{ by Lemma \ref{compare td}}\\
& = \max \{\hd_1 (\So V) + 1, \, \gd( \So P) + 1 \} \\
& \leqslant \max \{\hd_1 (\So V) + 1, \, \gd(P) + 1 \} & \text{ by Lemma \ref{compare h0}}\\
& = \max \{ \hd_1 (\So V) + 1, \, \hd_0 (V) + 1 \},
\end{align*}
as claimed by Statement (1).

One can use recursion to prove the second statement. Replacing $V$ by $V^1$ (which is also super finitely generated) and using the same argument, one deduces that
\begin{equation*}
\hd_1 (V^1) \leqslant \max \{ \hd_0 (V^1) + 1, \, \hd_1 (\So V^1) + 1 \}.
\end{equation*}
Note that $\hd_1 (\So V^1) = \hd_2 (\So V)$, and by the previous lemma,
\begin{equation*}
\hd_0 (V^1) + 1 \leqslant \max \{ \hd_0(V) + 1, \, \hd_1(V) + 1 \}.
\end{equation*}
Putting these two inequalities together, one deduces that
\begin{equation*}
\hd_2 (V) = \hd_1(V^1) \leqslant \max \{ \hd_0(V) + 1, \, \hd_1 (V) + 1, \, \hd_2 (\So V) + 1 \}.
\end{equation*}
The conclusion follows from recursion.
\end{proof}

\subsection{Castelnuovo-Mumford regularity under shift functors}

It is often the case that we need to apply $\So$ several times to a given $\tC$-module. The following proposition, deduced by extensively using Proposition \ref{shift keeps sfp}, plays a key role for estimating upper bounds for homological degrees of many interesting modules including torsion modules.

\begin{proposition} \label{crucial proposition}
Let $V$ be a $\FPi$ $\tC$-module. Suppose that there exists a certain $a \in \Z$ such that
\begin{equation} \label{given condition}
\hd_s (\Sa V) \leqslant \hd_0 (\Sa V) + s.
\end{equation}
for $s \geqslant 0$. Then for $s \geqslant 0$ one also has
\begin{equation*}
\hd_s (V) \leqslant \hd_0 (V) + s + a.
\end{equation*}
\end{proposition}

\begin{proof}
We use induction on $s$. The conclusion for $s = 0$ holds trivially. Now suppose that the conclusion is true for all $s$ which is at most $n \in \Z$, and let us consider $s = n + 1$. One has
\begin{equation*}
\hd_{n+1} (V) \leqslant \max \{\hd_0(V) + 1, \, \ldots, \, \hd_n (V) +1, \, \hd_{n+1} (\So V) + 1 \}
\end{equation*}
by Proposition \ref{hd under shift}. By the induction hypothesis, for $0 \leqslant i \leqslant n$, one has
\begin{equation*}
\hd_i (V) + 1 \leqslant \hd_0(V) + a + i + 1 \leqslant \hd_0(V) + a + n + 1 = \hd_0(V) + a + s.
\end{equation*}
Therefore, it suffices to show that
\begin{equation} \label{first inequality}
\hd_{n+1} (\So V) \leqslant \hd_0 (V) + a + n.
\end{equation}

If $a = 1$, letting $s = n+1$ in the given inequality (\ref{given condition}) one has
\begin{equation*}
\hd_{n+1} (\So V) \leqslant \hd_0 (\So V) + n + 1 \leqslant \hd_0(V) + n + 1
\end{equation*}
by Lemma \ref{compare h0}, which is exactly what we want. Otherwise, note that $\So V$ also satisfies the inequality (\ref{given condition}) (replacing $a$ by $a - 1$ and $V$ by $\So V$). Applying (2) of Proposition \ref{hd under shift} to $\So V$ rather than $V$ one has:
\begin{equation*}
\hd_{n+1} (\So V) \leqslant \max \{\hd_0 (\So V) +1, \, \ldots, \, \hd_n (\So V) +1, \, \hd_{n+1} (S_2 V) + 1 \}
\end{equation*}
by Proposition \ref{hd under shift}. By the induction hypothesis on $\So V$, for $0 \leqslant i \leqslant n$, one has
\begin{equation*}
\hd_i (\So V) + 1 \leqslant \hd_0(\So V) + (a - 1) + i + 1 = \hd_0( \So V) + a + i \leqslant \hd_0(V) + a + n.
\end{equation*}
Therefore, to prove inequality (\ref{first inequality}), it suffices to show that
\begin{equation}
\hd_{n+1} (S_2 V) \leqslant \hd_0 (V) + a + n - 1.
\end{equation}

One can repeat the above argument recursively, and finally it suffices to verify
\begin{equation*}
\hd_{n+1} (\Sa V) \leqslant \hd_0 (V) + n + 1.
\end{equation*}
But this is implies by inequality (\ref{given condition}) since one always has $\hd_0 (\Sa V) \leqslant \hd_0(V)$.

We have proved the wanted inequality for $s = n + 1$ recursively. The conclusion then follows from induction.
\end{proof}

\begin{remark} \normalfont
The above proposition tells us that if there exists a certain $a \in \Z$ such that $\Sa V$ has Castelnuovo-Mumford regularity bounded by $\hd_0 (\Sa V)$, then $V$ has Castelnuovo-Mumford regularity bounded by $\hd_0 (V) + a$. In a forthcoming paper we will show that inequality (\ref{given condition}) is satisfied for every finitely generated $\FI$-module.
\end{remark}

\subsection{Homological degrees of torsion modules}

Now we consider homological degrees of torsion modules.

\begin{theorem} \label{hd of torsion modules}
If $V$ is a finitely generated torsion $\C$-module, then for $s \in \Z$, one has
\begin{equation*}
\hd_s(V) \leqslant \td(V) + s.
\end{equation*}
\end{theorem}

\begin{proof}
The conclusion trivially holds for $s = 0$ since we always have $\hd_0(V) \leqslant \td(V)$ for torsion modules. So we let $s \geqslant 1$. We use induction on $\td(V)$. Firstly, let us consider $\td(V) = 0$. Then $\Sigma V = 0$, and hence $\hd_s (\Sigma V) = - \infty$ for all $s \in \Z$. By Proposition \ref{hd under shift}, one has
\begin{equation*}
\hd_s (V) \leqslant \max \{\hd_0(V) + 1, \, \ldots , \, \hd_{s-1} (V) + 1, \, \hd_s (\Sigma V) + 1 \} =  \max \{\hd_0(V) + 1, \, \ldots , \, \hd_{s-1} (V) + 1\}.
\end{equation*}
Using this recursive formula, one easily sees that
\begin{equation*}
\hd_s (V) \leqslant s = s + \td(V)
\end{equation*}
for $s \in \Z$.

Now suppose that the conclusion holds for all torsion modules with torsion degrees at most $n$, and let $V$ be a torsion module such that $\td(V) = n + 1$. Clearly, $\td (\Sigma V) = n$. Therefore, by the induction hypothesis,
\begin{equation*}
\hd_s (\Sigma V) \leqslant s + \td (\Sigma V) = s + n
\end{equation*}
for $s \in \Z$. Consequently, Proposition \ref{hd under shift} tells us that
\begin{equation*}
\hd_s (V) \leqslant \max \{\hd_0(V) + 1, \, \ldots, \, \hd_{s-1} (V) + 1, \, s + n + 1 \}.
\end{equation*}
Recursively, we verify
\begin{align*}
& \hd_1 (V) \leqslant \{\hd_0(V) + 1, \, 1 + n + 1 \} = n + 2;\\
& \hd_2 (V) \leqslant \{\hd_0(V) + 1, \, \hd_1(V) + 1, \, 2 + n + 1 \} = n + 3
\end{align*}
and so on. That is, for $s \geqslant 0$,
\begin{equation*}
\hd_s (V) \leqslant s + n + 1 = s + \td(V)
\end{equation*}
as claimed. The conclusion follows from induction.
\end{proof}

\begin{remark} \normalfont
These upper bounds for homological degrees of finitely generated torsion modules are not optimal; See a detailed discussion in Example \ref{example}.
\end{remark}

The above theorem applies to many other $\FPi$ modules such as truncations of projective modules.

\begin{corollary}
Let $P$ be a finitely generated projective $\tC$-module and $n$ be a nonnegative integer. Then
\begin{equation*}
\hd_s (\tau_n P) \leqslant n + s.
\end{equation*}
for $s \geqslant 1$.
\end{corollary}

\begin{proof}
Note that if $n < \gd(P)$, then $\tau_n P \cong Q \oplus V$ where $Q$ is a nonzero projective $\tC$-module and $V$ is either 0 or a $\tC$-module with $\gd(V) = n$. In this case, for $s \geqslant 1$, we have $\hd_s(\tau_n P) = \hd_s(V)$. Therefore, without loss of generality one can assume that $n \geqslant \gd(P)$. Consequently, $\gd (\tau_n P) = n$.

Consider the short exact sequence
\begin{equation*}
0 \to \tau_n P \to P \to \overline{P} \to 0,
\end{equation*}
where $\overline{P}$ is a finitely generated torsion module with $\td(\overline{P}) < n$. Therefore,
\begin{equation*}
\hd_s (\tau_n P) = \hd_{s+1} (\overline{P}) \leqslant \td (\overline{P}) + s + 1 < n + s + 1
\end{equation*}
as claimed.
\end{proof}

\section{Applications in representation stability theory}

In this section we apply general results obtained in previous sections to combinatorial categories appearing in representation stability theory.

\subsection{Categories with genetic functors}

In \cite{GL2} Gan and the author showed that the $\mk$-linearizations of the following combinatorial categories all have genetic functors:
\begin{equation} \label{list}
\FI_G, \quad \OI_G, \quad \VI, \quad \FI_d, \quad \OI_d, \quad \FS_G^{\mathrm{op}}, \quad \OS_G^{\mathrm{op}}.
\end{equation}
Therefore, applying Theorem \ref{hd of torsion modules}, one immediately gets

\begin{corollary} \label{hd of torsion modules of combinatorial categories}
Let $\tC$ be the $\mk$-linearization of one of the above combinatorial categories, and let $V$ be a finitely generated torsion $\tC$-module. Then
\begin{equation*}
\hd_s (V) \leqslant \td(V) + s
\end{equation*}
for $s \in \Z$.
\end{corollary}

\begin{remark} \normalfont
These categories have been shown to be locally Noetherian; see \cite{SS2}. Therefore, every finitely generated representation is $\FPi$. However, except for the category $\FI$ (see \cite[Theorem A]{CE}), no explicit upper bounds for homological degrees of finitely generated representations was described before, even for finitely generated torsion modules.
\end{remark}

\subsection{Koszul modules}

In the rest of this paper we let $\mk$ be a field of characteristic 0. For this subsection let $\tC$ be the $\mk$-linearization of one of the categories in the List (\ref{list}).
The following result was established in \cite[Proposition 2.10]{GL2}.

\begin{proposition}
For every $V \in \tC \fgmod$, projective cover of $V$ exists and is unique up to isomorphism.
\end{proposition}

The projective cover $P_V$ of $V$ has the following explicit description:
\begin{equation*}
P_V = \tC \otimes _{\tC_0} H_0 (V),
\end{equation*}
where we regard $\tC_0$ as a subcategory of $\tC$. Since $H_0(V) \cong V/JV$ is finitely generated, and each $\tC(i, i)$ is a finite dimensional semisimple algebra for $i \in \Z$, the reader easily sees that $P_V$ is indeed a finitely generated projective $\tC$-module. Moreover, one has $H_0 (P) \cong H_0(V)$.

Since projective cover of $V$ exists and is unique up to isomorphism, \emph{syzygies} are well defined. Explicitly, given a finitely generated $\tC$-module $V$, there exists a surjection
\begin{equation*}
\tC \otimes _{\tC_0} H_0(V) \to V.
\end{equation*}
The first syzygy $\Omega V$ is defined to be the kernel of this map, which is unique up to isomorphism. Recursively, one can define $\Omega ^i V$, the $i$-th syzygy of $V$ for $i \in \Z$.

For $V \in \tC \fgmod$, its \emph{support}, denoted by $\supp(V)$, is set to be
\begin{equation*}
\{ i \in \Z \mid V_i \neq 0 \}.
\end{equation*}
Its \emph{initial degree}, denoted by $\ini(V)$, is the minimal object in $\supp(V)$. If $V = 0$, we let $\ini(V) = - \infty$.

\begin{lemma} \label{ini}
Let $V$ be a finitely generated $\tC$-module. One has $H_1(V) \cong H_0 (\Omega V)$. In particular, $\hd_1(V) = \gd(\Omega V)$, and hence $\ini (\Omega V) > \ini(V)$ if $\Omega V \neq 0$.
\end{lemma}

\begin{proof}
Let $P$ be a projective cover of $V$. From the short exact sequence
\begin{equation*}
0 \to \Omega V \to P \to V \to 0
\end{equation*}
one gets
\begin{equation*}
0 \to H_1(V) \to H_0 (\Omega V) \to H_0(P) \to H_0(V) \to 0.
\end{equation*}
However, since $P$ is a projective cover of $V$, $H_0(P) \cong H_0(V)$. Consequently, $H_1(V) \cong H_0 (\Omega V)$ as claimed. It immediately follows that $\hd_1(V) = \gd (\Omega V)$. Moreover, since the values of $P$ and $V$ on the object $\ini(V)$ must be isomorphic, this forces the value of $\Omega V$ on $\ini(V)$ to be 0, so $\ini (\Omega V) > \ini(V)$ when $\Omega V \neq 0$.
\end{proof}

In \cite{GL2} we proved that $\tC$ is a Koszul category. That is, $\tC (i,i)$ has a linear projective resolution for every $i \in \Z$. We give an equivalent definition here; see \cite[Lemma 4.2]{GL2}. For a general introduction of Koszul theory (including several generalized versions), one may refer to \cite{BGS, Li, MOS}.

\begin{definition}
A finitely generated $\tC$-module generated in degree $d$ is said to be Koszul if for $s \geqslant 0$, one has $\supp(H_s (V)) \subseteq \{ s+d \}$. In other words, $\Omega^s V$ is either 0 or is generated in degree $s+d$.
\end{definition}

Using this equivalent definition, Corollary \ref{hd of torsion modules of combinatorial categories} actually gives another proof for the Koszulity of $\tC$.

\begin{corollary} \cite[Theorem 4.15]{GL2}
The category $\tC$ is a Koszul category.
\end{corollary}

\begin{proof}
For an arbitrary $i \in \Z$, let $V = \tC(i, i)$, regarded as a $\tC$-module concentrated on object $i$. By Theorem \ref{hd of torsion modules}, one has $\hd_s(V) \leqslant i + s$ for $s \in \Z$. Therefore,
\begin{equation*}
\supp(H_s (V)) \subseteq \{i, \, i+1, \, \ldots, \, i + s \}.
\end{equation*}
On the other hand, by Lemma \ref{ini}, if $H_s(V) \neq 0$, then
\begin{equation*}
H_s (V) = H_{s-1} (\Omega V) = \ldots = H_0 (\Omega^s V)
\end{equation*}
and
\begin{equation*}
\ini (\Omega^s V) > \ini (\Omega^{s-1} V) > \ldots > \ini( \Omega(V)) > \ini (V).
\end{equation*}
This forces
\begin{equation*}
\ini (H_s(V)) = \ini (H_0 (\Omega^s(V))) = \ini(\Omega^s V) \geqslant \ini(V) + s = i + s.
\end{equation*}
Consequently, $\supp (H_s(V) = \{ i+s \}$. By the above definition, $V = \tC(i, i)$ is a Koszul module. Since $i$ is arbitrary, $\tC$ is a Koszul category.
\end{proof}

The following proposition tells us that functors $\tau_n$ preserve Koszul modules for $n \in \Z$.

\begin{proposition} \label{Koszul modules}
Let $V$ be a finitely generated $\tC$-module. Then:
\begin{enumerate}
\item If $V$ is generated in degree $d$ with $d \geqslant 1$ and $\So V$ is Koszul, then $V$ is Koszul as well.
\item If $V$ is Koszul, so is $\tau_n V$ for every $n \in \Z$.
\end{enumerate}
\end{proposition}

\begin{proof}
The first part is precisely \cite[Proposition 4.13]{GL2}, although a different notation is used there.

Now we prove (2). Let $V$ be a nonzero Koszul module generated in degree $d$ for a certain $d \in \Z$. For $n < d$, the conclusion holds trivially since $\tau_n V \cong V$. So we assume $n \geqslant d$ and carry out induction on the difference $n - d$. The conclusion holds for $n - d = 0$. Suppose that it is true for $n - d = r$, and let us consider $n - d = r+1$.

Applying $\tC_0 \otimes _{\tC} -$ to the exact sequence
\begin{equation*}
0 \to \tau_n V \to \tau_{n-1} V \to V_{n-1} \to 0
\end{equation*}
of $\tC$-modules one gets
\begin{equation*}
\ldots \to H_{s+1} (V_{n-1}) \to H_s(\tau_n V) \to H_s (\tau_{n-1} V) \to H_s (V_{n-1}) \to \ldots.
\end{equation*}
Since $\tC$ is a Koszul category, $V_n$ is a Koszul module. By induction, $\tau_{n-1} V$ is also a Koszul module. Therefore,
\begin{equation*}
\supp (H_{s+1} (V_{n-1})) \subseteq \{ n - 1 + s + 1\} = \{n+s \}
\end{equation*}
and
\begin{equation*}
\supp (H_s (\tau_{n-1} V)) \subseteq \{ n - 1 + s \}.
\end{equation*}
Therefore,
\begin{equation*}
\supp (H_s (\tau_n V)) \subseteq \{ n - 1 + s, \, n + s \}.
\end{equation*}
However, by Lemma \ref{ini}, if $\Omega^s (\tau_n V) \neq 0$, one has
\begin{equation*}
\ini (\Omega^s (\tau_n V)) > \ini (\Omega ^{s-1} (\tau_n V)) > \ldots > \ini (\tau_n V) = n,
\end{equation*}
so $\ini (\Omega^s (\tau_n V)) \geqslant n+s$. However, since
\begin{equation*}
H_s (\tau_n V) \cong \tC_0 \otimes _{\tC} \Omega^s (\tau_n V),
\end{equation*}
one also has $\ini (H_s (\tau_n V)) \geqslant s+n$ if $\Omega^s V (\tau_n V)$ is nonzero. Thus $\supp (H_s (\tau_n V)) = \{ n + s \}$ for $s \in \Z$, and hence $\tau_n V$ is Koszul as well. The conclusion then follows from induction.
\end{proof}

\begin{remark} \normalfont
In the first statement of this proposition we require $d \geqslant 1$ to avoid the following case. Let $P = \tC(0,-)$ and $V = P / J^2 P$. Then $\So V$ is a Koszul module, but $V$ is not Koszul.
\end{remark}

\begin{remark} \normalfont
The reader can see that proofs in this subsection do not rely on any specific property of these combinatorial categories in List (\ref{list}). Therefore, all results described in this subsection hold for general $k$-linear, locally finite categories of type $\Ai$, provided that they are equipped with genetic functors, and the endomorphism algebra of each object is a finite dimensional semisimple algebra.
\end{remark}

\subsection{Homological degrees of $\FI$-modules}

In the rest of this paper let $\mk$ be a field of characteristic 0, and fix $\tC$ to be the $\mk$-linearization of $\FI$. We list some results which will be used later.

\begin{theorem}
Let $\tC$ be the $\mk$-linearization of $\FI$. Then:
\begin{enumerate}
\item The category $\tC \fgmod$ is abelian.
\item Every finitely generated projective $\tC$-module is also injective. Moreover, a finitely generated injective $\tC$-module is a direct sum of a finitely generated projective module and a finite dimensional injective module.
\item Every finitely generated $\tC$-module $V$ has a finite injective resolution. In particular, for a sufficiently large $a$, $\Sa V$ is a finitely generated projective module.
\end{enumerate}
\end{theorem}

\begin{remark} \normalfont
Statement (1) of the theorem was firstly proved by Church, Ellenberg, and Farb in \cite{CEF} over fields with characteristic 0.  In \cite{CEFN} they and Nagpal showed the same conclusion for arbitrary commutative Noetherian rings by using the shift functor. The result was generalized to many combinatorial categories by Gan and the author in \cite{GL1}, and by Sam and Snowden in \cite{SS2}.

Statement (2) and the first half of Statement (3) were proved by Sam and Snowden in \cite{SS1} in the language of twisted commutative algebras. Using the coinduction functor related to $\So$, Gan and the author gave in \cite{GL3} a proof of Statements (2) and (3) for $\FI_G$. In \cite{GL3} we also observed that (3) implies the representation stability of $\FI_G$-modules.
\end{remark}

\begin{remark} \normalfont
In \cite{GL3} Gan and the author proved that when $\mk$ is a field of characteristic 0, every finitely generated projective representation of $\VI$ is injective as well. However, we do not know whether the second half of Statement (2) and Statement (3) hold for $\VI$.
\end{remark}

\begin{lemma} \label{compare gd to hd}
Let $V$ be a nonzero finitely generated $\FI$-module. Then $\td(V) < \infty$. Moreover, if $V$ has no projective summands, then
\begin{equation*}
\gd(V) < \hd_1(V).
\end{equation*}
\end{lemma}

\begin{proof}
Since $\FI$ is locally Noetherian over a filed of characteristic 0, $V$ is a Noetherian module. This implies the first statement; see \cite[Definition 3.3.2, I. Injectivity]{CEF} or \cite[Proposition 5.1]{GL1}.

Since projective cover of $V$ exists, one may consider the exact sequence
\begin{equation*}
0 \to \Omega V \to P \to V \to 0.
\end{equation*}
Note that $\hd_1 (V) = \gd (\Omega V)$. Also note that $\gd(V) = \gd(P)$.

Let $s = \gd(V)$. The above short exact sequence gives rise to a commutative diagram of exact sequences
\begin{equation*}
\xymatrix{
0 \ar[r] & \Omega V \cap P' \ar[r] \ar[d] & P' \ar[r] \ar[d] & V' \ar[r] \ar[d] & 0\\
0  \ar[r] & \Omega V \ar[r] \ar[d] & P \ar[r] \ar[d] & V \ar[r] \ar[d] & 0\\
0 \ar[r] & \overline{\Omega V} \ar[r] & P'' \ar[r] & V'' \ar[r] & 0
}
\end{equation*}
where $P'$ is generated in degrees $<s$ and $P''$ is generated in degree $s$.

Since $\gd(V) = s$, $P''$ cannot be 0. Since $V$ has no projective summands, $\overline{ \Omega V}$ cannot be 0. Moreover, one has $P''_s \cong V''_s$ since $P''$ is also a projective cover of $V''$. Consequently,
\begin{equation*}
\hd_1 (V) = \gd (\Omega V) \geqslant \gd (\overline {\Omega V}) > \gd(P'') = s = \gd(V)
\end{equation*}
as claimed.
\end{proof}

Recall that every finitely generated $\FI$-module has a finite injective resolution. We briefly mention the construction; for details, see \cite[Section 7]{GL3}.

Let $V$ be a nonzero finitely generated $\FI$-module. Again, without loss of generality we can assume that $V$ has no projective summands since projective $\FI$-modules are also injective. The module $V$ gives rise to a short exact sequence
\begin{equation*}
0 \to V_T \to V \to V_F \to 0
\end{equation*}
such that $V_T$ is a finite dimensional torsion module (might be 0) and $V_F$ is a finitely generated torsionless module (might be 0). Now since $V_T$ is finite dimensional, one can get an injection $V_T \to I^0$ where $I^0$ is the injective hull of $V_T$. For $V_F$, by \cite[Proposition 7.5]{GL3}, there is also an injection $V_F \to P^0$ with $\gd (P^0) < \gd(V_F)$. Of course, we can make $P^0$ minimal by removing all projective summands from the cokernel of this map. Putting these two injections together, one has
\begin{equation} \label{construct injective resolution}
\xymatrix{
0 \ar[r] & V_T \ar[r] \ar[d] & V \ar[r] \ar[d] & V_F \ar[r] \ar[d] & 0\\
0 \ar[r] & I^0 \ar[r] \ar[d] & I^0 \oplus P^0 \ar[r] \ar[d] & P^0 \ar[r] \ar[d] & 0\\
0 \ar[r] & C' \ar[r] & \ar[r] V^{-1} \ar[r] & C'' \ar[r] & 0,
}
\end{equation}
where $C'$ and $C''$ are the cokernels. Repeating the above procedure for $V^{-1}$, after finitely many steps, one can reach $V^{-i} = 0$ at a certain step; see \cite[Theorem 1.7]{GL3}.

\begin{remark} \normalfont
Note that one cannot expect $C'$ to be precisely the torsion part of $V^{-1}$ because $C''$ might not be torsionless. Therefore, one has to construct the torsion-torsionless exact sequence again for $V^{-1}$, which in general is different from the bottom sequence in the above diagram.
\end{remark}

\begin{remark} \normalfont
In \cite{N} Nagpal showed that for an arbitrary commutative Noetherian ring $\mk$ and a finitely generated $\FI$-module $V$, it gives rise to a complex of finite length each term of which is a certain special module coinciding with a projective $\tC$-module when $\mk$ is a field of characteristic 0; see \cite[Theorem A]{N}. His result generalized our construction.
\end{remark}

The following observation is crucial for us to obtain an upper bound for homological degrees of $\tC$-modules.

\begin{lemma} \label{hd of cokernel}
Let $V$ and $V^{-1}$ be as above. If $V$ is torsionless, then for $s \in \Z$,
\begin{equation*}
H_s(V) \cong H_{s+1} (V^{-1}) \quad \text{and } \hd_s (V) = \hd_{s+1} (V^{-1}).
\end{equation*}
\end{lemma}

\begin{proof}
Since $V$ is torsionless, we know that $V_T = C' = 0$, $V_F \cong V$, and $V^{-1} \cong C''$ in the above diagram. Furthermore, $\Omega C'' \cong V$ since $V$ has been supposed to have no projective summands. Now the conclusion follows from Lemma \ref{ini}.
\end{proof}

The following result gives us an upper bound for homological degrees of torsionless modules.

\begin{lemma} \label{hd of torsionless modules}
Let $V$ be a torsionless $\tC$-module. Then for $s \geqslant 1$,
\begin{equation*}
\hd_s (V) \leqslant 2\gd(V) + s - 1.
\end{equation*}
\end{lemma}

\begin{proof}
The proof relies on \cite[Theorem A]{CE}. One may assume that $V$ has no projective summands. By Lemma \ref{hd of cokernel}, one has
\begin{equation*}
\hd_s(V) = \hd_{s+1} (V^{-1}) \leqslant \gd(V^{-1}) + \hd_1(V^{-1}) + s
\end{equation*}
by \cite[Theorem A]{CE}. Howover, from Diagram (\ref{construct injective resolution}), one has
\begin{equation*}
\hd_1(V^{-1}) = \gd(V)
\end{equation*}
and
\begin{equation*}
\gd(V^{-1}) = \gd(P^0) \leqslant \gd(V) - 1.
\end{equation*}
Consequently, we have
\begin{equation*}
\hd_s(V) \leqslant \gd(V^{-1}) + \hd_1(V^{-1}) + s \leqslant 2\gd(V) + s - 1
\end{equation*}
as claimed.
\end{proof}

Now we are ready to prove the main result of this section.

\begin{theorem}
Let $\mk$ be a field of characteristic 0 and $\tC$ be the $\mk$-linearization of $\FI$. Let $V$ be a finitely generated $\tC$-module. Then for $s \geqslant 1$, we have
\begin{equation*}
\hd_s (V) \leqslant \max \{ \td(V), \, 2\gd(V) - 1 \} + s.
\end{equation*}
\end{theorem}

\begin{proof}
The short exact sequence
\begin{equation*}
0 \to V_T \to V \to V_F \to 0
\end{equation*}
induces a long exact sequence
\begin{equation*}
\ldots \to H_s(V_T) \to H_s(V) \to H_s(V_F) \to \ldots.
\end{equation*}
We deduce that
\begin{equation*}
\hd_s(V) \leqslant \max \{ \hd_s(V_T), \, \hd_s(V_F) \}.
\end{equation*}
Note that
\begin{equation*}
\hd_s(V_T) \leqslant \td(V_T) + s = \td(V) + s
\end{equation*}
by Corollary \ref{hd of torsion modules of combinatorial categories} and
\begin{equation*}
\hd_s(V_F) \leqslant 2\gd(V_F) + s - 1 \leqslant 2\gd(V) + s - 1
\end{equation*}
by the previous lemma. The conclusion follows.
\end{proof}

\begin{remark} \normalfont
For torsionless modules, the upper bound provided in this theorem is always optimal than that of \cite[Theorem A]{CE}. Indeed, if $V$ is torsionless, we have
\begin{equation*}
\hd_s (V) \leqslant 2\gd(V) + s - 1 < \gd(V) + \hd_1(V) + s - 1
\end{equation*}
by Lemma \ref{compare gd to hd}.
\end{remark}

Here is an example:

\begin{example} \label{example} \normalfont
The projective $\tC$-module $\tC(1, -)$ has the following structure:
\begin{equation*}
\xymatrix{
 0 \ar[r] & \mk_1 \ar[r] \ar[dr] & \mk_2 \ar[r] \ar[dr] & \mk_3 \ar[r] \ar[dr] & \mk_4 \ar[r] \ar[dr] & \ldots \\
 & & \epsilon_2 \ar[r] & \epsilon_3 \ar[r] & \epsilon_4 \ar[r] & \ldots
}
\end{equation*}
where $\mk_i$ and $\epsilon_i$ are the trivial representation and standard representation of symmetric groups with $i$ letters. Let $V$ be the submodule
\begin{equation*}
\xymatrix{
 0 \ar[r] & 0 \ar[r] \ar[dr] & \mk_2 \ar[r] \ar[dr] & \mk_3 \ar[r] \ar[dr] & \mk_4 \ar[r] \ar[dr] & \ldots \\
 & & 0 \ar[r] & \epsilon_3 \ar[r] & \epsilon_4 \ar[r] & \ldots
}
\end{equation*}
and $\overline{V}$ be the quotient $\tC(1,-) / V$. A direct computation shows that $\gd(V) = 2$ and $\hd_1(V) = 4$. Therefore, by \cite[Theorem A]{CE}, one should have
\begin{equation*}
\hd_s(V) \leqslant 2 + 4 + s - 1 = s + 5
\end{equation*}
for $s \geqslant 1$. However, the above theorem tells us that
\begin{equation*}
\hd_s(V) \leqslant 4 + s - 1 = s + 3.
\end{equation*}
We also have $\td(\overline{V}) = 2 = \hd_1 (\overline{V})$. Therefore, the upper bounds described in Theorem \ref{hd of torsion modules} are not sharp.
\end{example}

\end{document}